\documentclass[12pt]{article}
\usepackage{amsmath,amsthm,amssymb}
\usepackage[top=30truemm, bottom=30truemm, left=25truemm, right=25truemm]{geometry}

\newtheorem{defi}{Definition}[section]
\newtheorem{lem}[defi]{Lemma}
\newtheorem{prop}[defi]{Proposition}
\newtheorem{thm}[defi]{Theorem}
\newtheorem{ex}[defi]{Example}

\usepackage{amsrefs}

\DeclareMathOperator{\rank}{rank}

\title{
Association schemes on
the Schubert cells of a Grassmannian
}
\author{Yuta Watanabe}
\date{\today}

\begin{document}
\maketitle

\begin{abstract}
Let $\mathbb{F}$ be any field.
The Grassmannian $\mathrm{Gr}(m,n)$ is the set of 
$m$-dimensional subspaces in $\mathbb{F}^n$, 
and the general linear group $\mathrm{GL}_n(\mathbb{F})$ acts transitively on it.
The Schubert cells of $\mathrm{Gr}(m,n)$ are the orbits of 
the Borel subgroup $\mathcal{B} \subset \mathrm{GL}_n(\mathbb{F})$ 
on $\mathrm{Gr}(m,n)$.
We consider the association scheme on each Schubert cell
defined by the $\mathcal{B}$-action
and show it is symmetric and it is
the \emph{generalized wreath product} of 
one-class association schemes,
which was introduced by
R.~A.~Bailey~[European Journal of Combinatorics 27 (2006) 428--435].
\bigskip

\end{abstract}

\section{Introduction}

Let $n$ be a positive integer, 
and let $\mathbb{F}$ be any field.
The \emph{subspace lattice} $\mathcal{P}_n(\mathbb{F})$
(also known as the \emph{finite projective space})
is the poset of linear subspaces of $\mathbb{F}^n$.
The general linear group $\mathrm{GL}_n(\mathbb{F})$ 
acts on $\mathcal{P}_n(\mathbb{F})$.
The natural grading structure
of $\mathcal{P}_n(\mathbb{F})$
is given by $\mathrm{GL}_n(\mathbb{F})$-action,
and
each fiber (i.e., orbit) is called a \emph{Grassmannian}.
In other words,
the Grassmannian $\mathrm{Gr}(m,n)$ is the set of 
$m$-dimensional subspaces in $\mathbb{F}^n$, where $0 \le m \le n$.
Let $\mathcal{B}$ denote the Borel subgroup of $\mathrm{GL}_n(\mathbb{F})$.
Then, the $\mathcal{B}$-action defines 
a finer ``hyper-cubic'' grading structure
of $\mathcal{P}_n(\mathbb{F})$,
and
each fiber contained in $\mathrm{Gr}(m,n)$
is called a \emph{Schubert cell of $\mathrm{Gr}(m,n)$}.
See \cite{MR2474907} for details.
The author showed in \cite{W}
that
the algebra defined from the ``hyper-cubic'' grading structure of 
$\mathcal{P}_n(\mathbb{F})$ together with its incidence structure
has a close relation to the quantum affine algebra
$U_q(\widehat{\mathfrak{sl}}_2)$, if $\mathbb{F}$ is a finite field of $q^2$ elements.
In this paper,
we study the Schubert cells of a Grassmannian 
from the combinatorial point of view of \emph{association schemes}.
More precisely,
we show that the association scheme
defined by the $\mathcal{B}$-action
on each Schubert cell is a 
\emph{generalized wreath product} of 
one-class association schemes
with the base set $\mathbb{F}$.
The concept of 
a generalized wreath product
of association schemes
was introduced by
R.~A.~Bailey~\cite{MR2206477} in 2006.
The (usual) wreath product
of association schemes has
been actively studied (see e.g.,
\cite{MR2712017, MR2610292, MR2002219, MR2964723, MR3612420, MR2747797, MR3047011, MR2802177}),
and we may view the result of this paper as demonstrating 
the fundamental importance of Bailey's generalization as well.

Before the main discussion,
we briefly recall the notion of the generalized wreath product of association schemes.
For the definition of association schemes,
see \cite{MR882540, MR2535398, MR2184345},
and
for the theory of posets,
see \cite{MR2868112}.
Let $(X,\le)$ be a nonempty finite poset.
A subset $Y$ in $X$ is called an \emph{anti-chain}
if any two elements in $Y$ is incomparable.
For an anti-chain $Y$ in $X$,
define the \emph{down-set} (also known as the \emph{order ideal}) by 
\[
\mathrm{Down}(Y) = \lbrace x \in X \mid \text{$x < y$ for some $y \in Y$}\rbrace.
\]
Note that this definition follows \cite{MR2206477}
and it is different from \cite{MR2868112},
where 
$\mathrm{Down}(Y) \cup Y$ is called 
the down-set of $Y$.
For each $x \in X$,
let $\mathcal{Q}_x$ denote an $r_x$-class association scheme
on a set $\Omega_x$.
We do not assume either $\mathcal{Q}_x$ is symmetric
or $\Omega_x$ is finite.
Let $R_{x,i}$ denote the $i$-th associate class
for $i \in \lbrace 0,1,\ldots,r_x\rbrace$.
By convention,
we choose the index 
so that $R_{x,0} = \lbrace (\omega, \omega) \mid \omega \in \Omega_x\rbrace$.
We set $\Omega = \prod_{x \in X} \Omega_x$.
For each anti-chain $Y$ in $X$
and for each $(i_x)_{x \in Y} \in \prod_{x \in Y} \lbrace 1,2,\ldots,r_x\rbrace$,
let
$R(Y,(i_x)_{x \in Y})$ denote the set of
$((\alpha_x)_{x \in X}, (\beta_x)_{x \in X}) \in \Omega \times \Omega$
satisfying
(i) $\alpha_x = \beta_x$ if $x \in X \setminus (Y \cup \mathrm{Down}(Y))$,
and
(ii) $(\alpha_x, \beta_x) \in R_{x,i_x}$ if $x \in Y$.
Let $\mathcal{R}$ denote the set of 
$R(Y,(i_x)_{x \in Y})$ for all anti-chains $Y$ in $X$ and 
$(i_x)_{x \in Y} \in \prod_{x \in Y} \lbrace 1,2,\ldots,r_x\rbrace$.

\begin{thm}[cf. {\cite[Theorem 3]{MR2206477}}]\label{Bailey}
The pair $(\Omega, \mathcal{R})$ is an association scheme.
\end{thm}

The association scheme $(\Omega, \mathcal{R})$ in Theorem \ref{Bailey} is called the 
\emph{generalized wreath product of $\mathcal{Q}_x$
over the poset $X$}.
We remark that Bailey~\cite[Theorem 3]{MR2206477}
assumes that each base set $\Omega_x$ is finite
and each association scheme $\mathcal{Q}_x$ is symmetric.
However, the theorem is still true if 
we drop both of these assumptions.

\section{Subspace lattices}

Throughout this paper, we fix a positive integer $n$ and a field $\mathbb{F}$.
Let $\mathbb{F}^n$ denote the $n$-dimensional column vector space over $\mathbb{F}$.
By the \emph{subspace lattice}, denoted by 
$\mathcal{P}_n(\mathbb{F})$,
we mean the poset consisting of all subspaces
in $\mathbb{F}^n$ with partial order given by inclusion.
We fix 
the sequence $\lbrace V_i \rbrace_{i = 0}^n$ in $\mathcal{P}_n(\mathbb{F})$ such that
each $V_i$ consists of vectors whose
bottom $n-i$ entries are zero.
We remark that 
$\lbrace V_i \rbrace_{i = 0}^n$ is a \emph{(complete) flag} (i.e., a maximal chain) in $\mathcal{P}_n(\mathbb{F})$.

Let $\mathrm{Mat}_n(\mathbb{F})$
denote the set of $n \times n$ matrices with entries in $\mathbb{F}$.
Let $\mathrm{GL}_n(\mathbb{F})$ denote the set of 
invertible matrices in $\mathrm{Mat}_n(\mathbb{F})$.
Observe that 
$\mathrm{GL}_n(\mathbb{F})$ acts on $\mathbb{F}^n$
by left multiplication,
and hence it acts on $\mathcal{P}_n(\mathbb{F})$.
Let $\mathcal{B}$ denote the (Borel) subgroup in $\mathrm{GL}_n(\mathbb{F})$ stabilizing
$\lbrace V_i \rbrace_{i = 0}^n$.
In other words,
$\mathcal{B}$ consists of all 
upper triangular invertible matrices in $\mathrm{Mat}_n(\mathbb{F})$.
In this paper,
we consider the $\mathcal{B}$-action on $\mathcal{P}_n(\mathbb{F})$.

A matrix in $\mathrm{Mat}_n(\mathbb{F})$
is said to be in \emph{reverse column echelon form}
if the following two conditions are met:
\begin{enumerate}
\item[(CE1)] Any zero columns are right of all nonzero columns.
\item[(CE2)] The last nonzero entry of a nonzero column is always
strictly below of the last nonzero entry of its right column.
\end{enumerate}
A matrix in $\mathrm{Mat}_n(\mathbb{F})$
is said to be in \emph{reduced reverse column echelon form}
if it is in reverse column echelon form and 
the following third condition is also met:
\begin{enumerate}
\item[(CE3)] Every last nonzero entry of a nonzero column is $1$ and 
is the only nonzero entry in its row.
\end{enumerate}

By elementary linear algebra, we have the following.
\begin{prop}\label{prop:bij}
There exists a bijection between the following two sets:
\begin{enumerate}
\item the set of all matrices in $\mathrm{Mat}_n(\mathbb{F})$ 
in reduced reverse column echelon form,
\item the set $\mathcal{P}_n(\mathbb{F})$ of all subspaces in $\mathbb{F}^n$,
\end{enumerate}
that sends a matrix in $\mathrm{Mat}_n(\mathbb{F})$
to its column space.
\end{prop}
\begin{proof}
See for instance \cite{MR3013937}.
\end{proof}

\begin{ex}[$n=7$]\label{ex1}
Let $e_1, e_2, \ldots, e_7$ denote the standard basis for $\mathbb{F}^7$.
Suppose $U$ is the $4$-dimensional subspace
in $\mathbb{F}^7$ given by
$U = \mathrm{Span}\lbrace
8e_1+6e_3+4e_6+2e_7,
8e_1+9e_3+e_4+e_5,
4e_1+e_2+5e_3+e_4,
3e_1+e_2\rbrace$.
Then the following is the matrix corresponding to $U$
by the bijection in Proposition \ref{prop:bij}:
\[
M = 
\begin{pmatrix}
4 & 7 & 1 & 3 & 0 & 0 & 0 \\
0 & 0 & 0 & 1 & 0 & 0 & 0 \\
3 & 4 & 5 & 0 & 0 & 0 & 0 \\
0 & 0 & 1 & 0 & 0 & 0 & 0 \\
0 & 1 & 0 & 0 & 0 & 0 & 0 \\
2 & 0 & 0 & 0 & 0 & 0 & 0 \\
1 & 0 & 0 & 0 & 0 & 0 & 0
\end{pmatrix}.
\]
\end{ex}
Observe that
for 
$M, N \in \mathrm{Mat}_n(\mathbb{F})$
in reduced reverse column echelon form
and for $G \in \mathcal{B}$,
the column space of $N$ moves to 
that of $M$ by the $G$-action
if and only if 
$M$ and $GN$ are column equivalent.
Since $GN$ is in reverse column echelon form
(but not necessarily reduced),
these conditions are equivalent to
$M = GNH^{T}$ for some $H \in \mathcal{B}$.
For notational convenience, 
we write $M \sim N$
if there exist $G, H \in \mathcal{B}$ such that $M = GNH^T$.
Observe that $\sim$ is an equivalence relation on $\mathrm{Mat}_n(\mathbb{F})$.

For the rest of this paper,
we will identify $\mathcal{P}_n(\mathbb{F})$ with the set of 
all matrices in $\mathrm{Mat}_n(\mathbb{F})$ in reduced reverse column echelon form
by the bijection in Proposition \ref{prop:bij}.

\section{The $\mathcal{B}$-action on $\mathcal{P}_n(\mathbb{F})$}

For a positive integer $m$, we write 
$[m] = \lbrace 1,2,\ldots, m\rbrace$.
We define a partial order in
the index set $[n] \times [n]$ of matrices in $\mathrm{Mat}_n(\mathbb{F})$ by 
$(i,j) \le (k,l)$ if $i \le k$ and $j \le l$.
This is known as the \emph{direct product order}
in \cite[Section~3.2]{MR2868112}.
For $M \in \mathrm{Mat}_n(\mathbb{F})$,
by the \emph{support} of $M$, denoted by $\mathrm{Supp}(M)$, we mean 
the subposet of $[n] \times [n]$
consisting of all indices $(i,j) \in [n] \times [n]$ with $M_{i,j} \neq 0$.
The \emph{pivot-set} of $M$,
denoted by $\mathrm{Piv}(M)$,
is the set of all maximal elements in $\mathrm{Supp}(M)$.
Each element in the pivot-set is called a \emph{pivot}.
Observe that
$(i,j) \in \mathrm{Piv}(M)$
if and only if
$M_{i,j} \neq 0$ and $M_{k,l} = 0$ if $(k,l) > (i,j)$.
We remark that
every entry indexed by a pivot of a matrix in $\mathcal{P}_n(\mathbb{F})$ 
must be $1$ by the condition (CE3).

\begin{lem}\label{lem:wequiv}
For $M, N \in \mathcal{P}_n(\mathbb{F})$,
the following are equivalent:
\begin{enumerate}
\item $M \sim N$,
\item $\mathrm{Piv}(M) = \mathrm{Piv}(N)$.
\end{enumerate}
\end{lem}
\begin{proof}
(i) $\Rightarrow$ (ii)
Suppose $M \sim N$.
There exist
$G, H \in \mathcal{B}$ such that $M = GNH^T$.
It suffices to show $\mathrm{Piv}(N) \subseteq \mathrm{Piv}(M)$.
For $(i,j) \in \mathrm{Piv}(N)$,
we have
$N_{i,j} = 1$ and $N_{k,l} = 0$ if $(k,l) > (i,j)$.
Since
$G, H$ are upper triangular,
we have
\[
M_{k,l} = \sum_{s = k}^n\sum_{t = l}^n G_{k,s}N_{s,t}H_{l,t}
=
\begin{cases}
G_{i,i}H_{j,j} & \text{if $(k,l) = (i,j)$},\\
0 & \text{if $(k,l) > (i,j)$}.
\end{cases}
\]
Since $G, H$ are invertible,
$G_{i,i}H_{j,j} \neq 0$.
These imply $(i,j) \in \mathrm{Piv}(M)$
and hence $\mathrm{Piv}(N) \subseteq \mathrm{Piv}(M)$.

(ii) $\Rightarrow$ (i)
Suppose $\mathrm{Piv}(M) = \mathrm{Piv}(N)$.
Take $X \in \mathcal{P}_n(\mathbb{F})$ with
$X_{i,j} = 1$ if $(i,j) \in \mathrm{Piv}(M)$ and $X_{i,j} = 0$ otherwise.
Observe that
for each $j \in [n]$,
there exists at most one $k$ such that $(j,k) \in \mathrm{Piv}(M)$
and then we
define $G \in \mathrm{Mat}_n(\mathbb{F})$ by
\[
G_{i,j} =
\begin{cases}
M_{i,k} & \text{if $(j,k) \in \mathrm{Piv}(M)$ for some $k$},\\
\delta_{i,j} & \text{if there is no $k$ such that  $(j,k) \in \mathrm{Piv}(M)$}
\end{cases}
\]
for $i,j \in [n]$.
Then we have $G_{i,i} = 1$ for $i \in [n]$
and $G_{i,j} = 0$ if $j < i$ for $i, j \in [n]$.
Thus 
$G \in \mathcal{B}$.
By the direct calculation, we have $M = GX$
and 
hence, $M \sim X$.
Similarly we have $N \sim X$ and so $M \sim N$.
\end{proof}

Let $1 \le m \le n-1$ and $M \in \mathcal{P}_n(\mathbb{F})$
with $\rank M = m$.
Note that 
we avoid the trivial cases $m = 0$ and $m = n$.
Since the pivots of $M$ lie in the first $m$ columns,
$\mathrm{Piv}(M)$
is an anti-chain in $[n] \times [m]$ of size $m$.
For $1 \le m \le n-1$
and for an anti-chain $\alpha$ in $[n] \times [m]$
 of size $m$,
we set
\begin{equation}
\mathcal{O}_{\alpha} = \lbrace M \in \mathcal{P}_n(\mathbb{F}) \mid
\mathrm{Piv}(M) = \alpha
\rbrace.
\label{OA}
\end{equation}
For each $1 \le m \le n-1$ and
each anti-chain $\alpha$ in $[n] \times [m]$ of size $m$,
consider
$M \in \mathcal{P}_n(\mathbb{F})$ with
$M_{i,j} = 1$ if $(i,j) \in \alpha$ and $M_{i,j} = 0$ otherwise.
Then we have $M \in \mathcal{O}_\alpha$
and in particular $\mathcal{O}_\alpha \neq \emptyset$.

\begin{prop}\label{prop:OA}
The rank of any matrix in \eqref{OA} 
is $m = |\alpha|$.
Moreover,
each subset \eqref{OA}
is an orbit of the $\mathcal{B}$-action on $\mathcal{P}_n(\mathbb{F})$.
\end{prop}
\begin{proof}
Immediate from the construction and Lemma \ref{lem:wequiv}.
\end{proof}

Recall the Grassmannian $\mathrm{Gr}(m,n)$
and we identify $\mathrm{Gr}(m,n)$ with a set of matrices
by the bijection in Proposition \ref{prop:bij}.
In other words,
\[
\mathrm{Gr}(m,n) = \lbrace M \in \mathcal{P}_n(\mathbb{F}) \mid
\rank M = m
\rbrace.
\]
By Proposition \ref{prop:OA},
each $\mathcal{O}_\alpha$ in \eqref{OA} is
a $\mathcal{B}$-orbit in $\mathrm{Gr}(m,n)$,
where $m = |\alpha|$.
Thus,
it is called a \emph{Schubert cell of a Grassmannian} \cite{MR2474907}.

\begin{ex}[$n=7$, $m = 4$]\label{ex2}
Take $M \in \mathcal{P}_7(\mathbb{F})$ as in Example \ref{ex1}.
Then we have
$\mathrm{Piv}(M) = \lbrace
(2,4), (4,3), (5,2), (7,1)
\rbrace$.
Moreover,
$\mathcal{O}_{\mathrm{Piv}(M)}$ is
the set of matrices of the form
\begin{equation}
\begin{pmatrix}
\ast & \ast & \ast & \ast & 0 & 0 & 0 \\
0 & 0 & 0 & 1 & 0 & 0 & 0 \\
\ast & \ast & \ast & 0 & 0 & 0 & 0 \\
0 & 0 & 1 & 0 & 0 & 0 & 0 \\
0 & 1 & 0 & 0 & 0 & 0 & 0 \\
\ast & 0 & 0 & 0 & 0 & 0 & 0 \\
1 & 0 & 0 & 0 & 0 & 0 & 0
\end{pmatrix},
\label{ex:Oalpha}
\end{equation}
where the symbol $\ast$ denotes an arbitrary element
in $\mathbb{F}$.
\end{ex}

\begin{lem}\label{lem:diag}
Let $1 \le m \le n-1$
and let $\alpha$ denote an anti-chain in $[n] \times [m]$ of size $m$.
For $M, N, M', N' \in \mathcal{O}_\alpha$,
the following are equivalent:
\begin{enumerate}
\item $(M,N)$ moves to $(M',N')$ by the diagonal $\mathcal{B}$-action,
\item $\mathrm{Piv}(M-N) = \mathrm{Piv}(M'-N')$.
\end{enumerate}
\end{lem}
\begin{proof}
(i) $\Rightarrow$ (ii)
Suppose there exist $G, H, K \in \mathcal{B}$
such that $M' = GMH^T$ and $N' = GNK^T$.
Then we have
\begin{equation}
\mathrm{Piv}(GM) = \mathrm{Piv}(GN) = \alpha,
\label{*}
\end{equation}
\begin{equation}
\mathrm{Piv}(GM-GN) = \mathrm{Piv}(M-N)
\label{**}
\end{equation}
since $G \in \mathcal{B}$ (cf. Lemma \ref{lem:wequiv}).
We write $\alpha = \lbrace(k_r,r) \mid r \in [m] \rbrace$
and observe that $k_1 > k_2 > \cdots > k_m$
and that $\mathrm{Piv}(M-N) \subseteq \mathrm{Down}(\alpha)$.

Take $(i,j) \in \mathrm{Piv}(M-N)$.
Observe that $j \in [m]$ and $k_j > i$
and hence 
there exists $m' = \max\lbrace r \in [m] \mid k_r > i\rbrace$.
For $r, l \in [m]$ with $j \le r \le m'$ and $j \le l \le m'$,
since $H, K$ are upper triangular,
we have
\begin{align*}
M'_{k_r,l} &= \sum_{t = l}^n (GM)_{k_r,t}H_{l,t}, \\
N'_{k_r,l} &= \sum_{t = l}^n (GN)_{k_r,t}K_{l,t}.
\end{align*}
In the above equations, we have the following:
For each $r,l$, we have $M'_{k_r,l} = N'_{k_r,l} = \delta_{r,l}$ by (CE3);
For each $r,t$, we have $(GM)_{k_r,t} = (GN)_{k_r,t}$
since $(k_r,t) > (i,j)$ and by \eqref{**};
For each $r, t$ with $r < t$, we have $(GM)_{k_r,r} = (GN)_{k_r,r} = G_{k_r,k_r} \neq 0$ 
and $(GM)_{k_r,t} = (GN)_{k_r,t} = 0$ 
since $(k_r,r) \in \alpha$ and by \eqref{*}.
By these comments,
for each $l \in [m]$ with $j \le l \le m'$,
both $(H_{l,l}, H_{l,l+1}, \ldots, H_{l,m'})$
and $(K_{l,l}, K_{l,l+1}, \ldots, K_{l,m'})$
are solutions to the same system of $m'-j+1$ independent linear equations.
Hence, $H_{l,t} = K_{l,t}$ for $j,t \in [m]$ with 
$j \le l \le t \le m'$.

For $(k,l) \in [n] \times [n]$ with $(k,l) > (i,j)$,
we have
$(GM)_{k,t} = (GN)_{k,t}$ if $t \ge l$ since $(k,t) > (i,j)$ and by \eqref{**},
and we also have 
$(GM)_{k,t} = (GN)_{k,t} = 0$ if $t > m'$ by the definition of $m'$ and by \eqref{*}.
Recall that we have shown $H_{l,t} = K_{l,t}$ if $j \le l \le t \le m'$.
By these comments, we have
\begin{align*}
M'_{k,l} = \sum_{t = l}^{n} (GM)_{k,t}H_{l,t}
= \sum_{t = l}^{n} (GN)_{k,t}K_{l,t} = N'_{k,l}.
\end{align*}
Similarly,
we have
\begin{align*}
M'_{i,j} - N'_{i,j} = \sum_{t = j}^{n} (GM)_{i,t}H_{j,t}
- \sum_{t = j}^{n} (GN)_{i,t}K_{j,t} 
= ((GM)_{i,j} - (GN)_{i,j})H_{j,j}.
\end{align*}
In the above equations,
we have $(GM)_{i,j} \neq (GN)_{i,j}$ by \eqref{**},
and we also have $H_{j,j} \neq 0$ since $H \in \mathcal{B}$.
Therefore we obtain $M'_{i,j} \neq N'_{i,j}$.
These imply $(i,j) \in \mathrm{Piv}(M'-N')$
and hence $\mathrm{Piv}(M-N) \subseteq \mathrm{Piv}(M'-N')$.
Since $G, H, K$ are invertible,
we also have 
$\mathrm{Piv}(M'-N') \subseteq \mathrm{Piv}(M-N)$.
Consequently, we have
$\mathrm{Piv}(M-N) = \mathrm{Piv}(M'-N')$.

(ii) $\Rightarrow$ (i)
Let
$M, N, M', N' \in \mathcal{O}_\alpha$
with $\mathrm{Piv}(M-N) = \mathrm{Piv}(M'-N')$.
Take $X \in \mathcal{O}_\alpha$ with
$X_{i,j} = 1$ if $(i,j) \in \alpha \cup \mathrm{Piv}(M-N)$ and $X_{i,j} = 0$ otherwise,
and $Y \in \mathcal{O}_\alpha$ with
$Y_{i,j} = 1$ if $(i,j) \in \alpha$ and $Y_{i,j} = 0$ otherwise.
Define $G \in \mathrm{Mat}_n(\mathbb{F})$ by
\[
G_{i,j} = \begin{cases}
N_{i,k} & \text{if $(j,k) \in \alpha$ for some $k$},\\
M_{i,k} - N_{i,k} & \text{if $(j,k) \in \mathrm{Piv}(M-N)$ for some $k$}, \\
\delta_{i,j} & \text{if there is no $k$ such that $(j,k) \in \alpha \cup \mathrm{Piv}(M-N)$}
\end{cases}
\]
for $i,j \in [n]$.
Then we have $G \in \mathcal{B}$
and $M = GX$ and $N = GY$.
Similarly there exists $G' \in \mathcal{B}$
such that
$M' = G'X$ and $N' = G'Y$.
Therefore
$(M,N)$ moves to $(M',N')$ by the diagonal action of 
$G'G^{-1}$.
\end{proof}

Let $1 \le m \le n-1$
and let $\alpha$ denote an anti-chain in $[n] \times [m]$ of size $m$.
Let $M, N \in \mathcal{O}_\alpha$.
Observe that
$\mathrm{Piv}(M-N)$
is an anti-chain in 
\begin{equation}
\mathcal{D}(\alpha)
=
\lbrace (i,j) \in \mathrm{Down}(\alpha) \mid
\text{there is no $k$ such that $(i,k) \in \alpha$}\rbrace.
\label{Dalpha}
\end{equation}
For $1 \le m \le n-1$
and for an anti-chain $\alpha$ in $[n] \times [m]$ of size $m$
and for an anti-chain $\beta$ in $\mathcal{D}(\alpha)$,
we set
\begin{equation}
\mathcal{R}_{\alpha,\beta} = \lbrace (M,N) \in \mathcal{O}_\alpha \times \mathcal{O}_\alpha \mid 
\mathrm{Piv}(M-N) = \beta
\rbrace.
\label{Rab}
\end{equation}
For each $1 \le m \le n-1$ and
each anti-chain $\alpha$ in $[n] \times [m]$ of size $m$
and each anti-chain $\beta$ in $\mathcal{D}(\alpha)$,
consider
$M \in \mathcal{P}_n(\mathbb{F})$ with
$M_{i,j} = 1$ if $(i,j) \in \alpha \cup \beta$ and $M_{i,j} = 0$ otherwise,
and $N \in \mathcal{P}_n(\mathbb{F})$ with
$N_{i,j} = 1$ if $(i,j) \in \alpha$ and $N_{i,j} = 0$ otherwise.
Then we have $(M, N) \in \mathcal{R}_{\alpha,\beta}$
and in particular $\mathcal{R}_{\alpha,\beta} \neq \emptyset$.

\begin{prop}\label{prop:Rab}
Let $1 \le m \le n-1$
and let $\alpha$ denote an anti-chain in $[n] \times [m]$ of size $m$.
Each subset \eqref{Rab}
is an orbital of the $\mathcal{B}$-action on $\mathcal{O}_\alpha$.
\end{prop}
\begin{proof}
Immediate from Lemma \ref{lem:diag}.
\end{proof}

Let $1 \le m \le n-1$.
For an anti-chain $\alpha$ in $[n] \times [m]$ of size $m$,
consider
\begin{equation}
\mathcal{D}_1(\alpha) = \lbrace i \mid \text{there is no $k$ such that $(i,k) \in \alpha$}\rbrace.
\label{D1alpha}
\end{equation}
Then we have $|\mathcal{D}_1(\alpha)| = n-m$ since 
$|\alpha| = m$.  
For $1 \le i \le n-m$,
we define
$\lambda_i = |\lbrace j \mid (d_i,j) \in \mathcal{D}(\alpha)\rbrace|$,
where $d_i$ denotes the $i$-th smallest element in $\mathcal{D}_1(\alpha)$.
Then $\lambda = (\lambda_1, \lambda_2, \ldots, \lambda_{n-m}) \in \mathbb{N}^{n-m}$ is 
an integer partition (i.e., a non-increasing sequence) with largest part at most $m$,
where 
\[
\mathbb{N} = \lbrace 0,1,\ldots\rbrace.
\]
Consider the map $\varphi_m$ which sends $\alpha$
to $\lambda$.

For an integer partition $\lambda = (\lambda_1, \lambda_2, \ldots, \lambda_l)$,
the \emph{Ferrers board} of shape $\lambda$
is defined by
\[
\lbrace (i,j) \in \mathbb{N} \times \mathbb{N} \mid 1 \le i \le l, 1 \le j \le \lambda_i \rbrace.
\]
We endow the Ferrers board
with direct product order in $\mathbb{N} \times \mathbb{N}$.

\begin{lem}\label{al}
For $1 \le m \le n-1$,
the map $\varphi_m$ is a bijection between the following two sets:
\begin{enumerate}
\item the anti-chains in $[n] \times [m]$ of size $m$;
\item the integer partitions in $\mathbb{N}^{n-m}$ with largest part at most $m$.
\end{enumerate}
\end{lem}
\begin{proof}
Let $1 \le m \le n-1$.
It is clear that $\varphi_m$ is a map from (i) to (ii).
We define the map $\varphi_m'$ from (ii) to (i) as follows.
For a given integer partition $\lambda$ in $\mathbb{N}^{n-m}$ with largest part at most $m$,
we define
$\alpha$ as the set of maximal elements 
in the Ferrers board of shape $\mu = \lambda \cup (m, m-1,\ldots,1)$,
which is the integer partition obtained by rearranging 
parts of both $\lambda$ and $(m, m-1,\ldots,1)$ in non-increasing order.
Since $\mu$ is in $\mathbb{N}^{n}$ and its largest part is $m$,
$\alpha$ is an anti-chain in $[n] \times [m]$ of size $m$.
The map $\varphi_m'$ is defined to send $\lambda$
to $\alpha$.
By construction,
$\varphi_m$ and $\varphi_m'$ are inverses and hence bijections.
\end{proof}

\begin{lem}\label{bmu}
Let $1 \le m \le n-1$
and
let $\alpha$ denote an anti-chain in $[n] \times [m]$ of size $m$.
The poset
$\mathcal{D}(\alpha)$ in \eqref{Dalpha} is isomorphic to
the Ferrers board of shape $\varphi_m(\alpha)$.
Moreover,
there is a one-to-one correspondence
between
the anti-chains in $\mathcal{D}(\alpha)$ 
and the subpartitions of $\varphi_m(\alpha)$.
\end{lem}
\begin{proof}
Recall the set $\mathcal{D}_1(\alpha)$ in \eqref{D1alpha}.
Then define the map $\psi$
from
the Ferrers board of shape $\varphi_m(\alpha)$ 
to $\mathcal{D}(\alpha)$ by
$\psi(i,j) = (d_i,j)$,
where $d_i$ denotes the $i$-th smallest element in $\mathcal{D}_1(\alpha)$.
It is obvious that $\psi$ is an order-preserving bijection.
The first assertion follows.
To show the second assertion,
we define the map $\rho$ from 
the subpartitions of $\varphi_m(\alpha)$
to
the anti-chains in $\mathcal{D}(\alpha)$
by 
$\rho(\mu) = \psi(\max(\mu))$,
where $\max(\mu)$ is the set of all maximal elements in 
the Ferrers board of shape $\mu$.
From the construction, 
the map $\rho$ is also a bijection.
The second assertion follows.
\end{proof}

\begin{ex}[$n=7$, $m=4$]
Take the anti-chain $\alpha = \lbrace
(2,4), (4,3), (5,2), (7,1)
\rbrace$ as in Example \ref{ex2}.
Recall that $\mathcal{O}_\alpha$ is the set of 
matrices of the form \eqref{ex:Oalpha}.
Then $\varphi_4(\alpha) = (4,3,1)$.
We remark that
each number in $(4,3,1)$
equals the number of $\ast$'s in each row without a pivot.
\end{ex}

\section{The association scheme on each Schubert cell}

For $1 \le m \le n-1$
and for an anti-chain $\alpha$ in $[n] \times [m]$ of size $m$,
by Propositions \ref{prop:OA} and \ref{prop:Rab},
the pair
\begin{equation}
\mathfrak{X}_\alpha = (\mathcal{O}_\alpha, \lbrace\mathcal{R}_{\alpha, \beta}\rbrace_\beta)
\label{Xalpha}
\end{equation}
becomes an association scheme,
where $\beta$ runs over all anti-chains in $\mathcal{D}(\alpha)$ in \eqref{Dalpha}.
See \cite[Preface]{MR2184345}.
We remark that by Lemma \ref{al},
the family of association schemes $\lbrace \mathfrak{X}_\alpha \rbrace_\alpha$
can be indexed by integer partitions $\lambda \in \mathbb{N}^{n-m}$ with largest part at most $m$.
In this case,
the associate classes of $\mathfrak{X}_\alpha = \mathfrak{X}_\lambda$ are indexed by 
the subpartitions of $\lambda$
by Lemma \ref{bmu}.
\begin{thm}
Let $1 \le m \le n-1$
and let $\alpha$ denote an anti-chain in $[n] \times [m]$ of size $m$.
The association scheme $\mathfrak{X}_\alpha$ in \eqref{Xalpha} is symmetric.
\end{thm}
\begin{proof}
Immediate from the definition of $\mathcal{R}_{\alpha, \beta}$.
\end{proof}

\begin{thm}
For $1 \le m \le n-1$
and for an anti-chain $\alpha$ in $[n] \times [m]$ of size $m$,
the association scheme $\mathfrak{X}_\alpha$ in \eqref{Xalpha} is 
the generalized wreath product of the one-class association schemes 
with the base set $\mathbb{F}$ over the poset $\mathcal{D}(\alpha)$ in \eqref{Dalpha}.
\end{thm}
\begin{proof}
For 
$M, N \in \mathcal{O}_\alpha$
and for
an anti-chain $\beta$ in $\mathcal{D}(\alpha)$,
we have
$(M,N) \in \mathcal{R}_{\alpha, \beta}$
if and only if
$M_{i,j} = N_{i,j}$ if $(i,j) \not\in \beta \cup \mathrm{Down}(\beta)$, $M_{i,j} \neq N_{i,j}$ if $(i,j) \in \beta$.
Therefore,
this associate relation
is the same as that of 
the generalized wreath product of the one-class association schemes 
with the base set $\mathbb{F}$ over the poset $\mathcal{D}(\alpha)$.
So the result follows.
\end{proof}

\section{Concluding remarks}

This paper focuses on the Schubert cells
of a Grassmannian.
It would be an interesting problem to find similar results on Schubert cells for other types of BN-pairs.

The Terwilliger algebra, introduced by P.~Terwilliger \cite{MR1203683}, 
of the wreath product of one-class association schemes
is discussed in several papers \cite{MR2610292, MR2747797, MR2802177}.
We will consider
the Terwilliger algebra of the generalized wreath product of one-class association schemes
in a future paper.

\section*{Acknowledgments}
The author gratefully acknowledges the many helpful suggestions of his advisor, Hajime Tanaka
during the preparation of the paper.
The author thanks Paul Terwilliger
for giving valuable comments
and Motohiro Ishii for drawing the author's attention to the theory of Schubert cells.

\bigskip

\noindent
Yuta Watanabe \\
Graduate School of Information Sciences \\
Tohoku University \\
Sendai, 980-8579 Japan\\
email: \texttt{watanabe@ims.is.tohoku.ac.jp}

\end{document}